\documentclass[11pt]{article}
\usepackage{amssymb}
\usepackage{amsthm}
\usepackage{amsmath}
\usepackage{enumerate}
\usepackage[american]{babel}
\usepackage{url}

\newtheorem{thm}{\noindent\hspace*{5mm}Theorem}
\newtheorem{lem}{\noindent\hspace*{5mm}Lemma}

\theoremstyle{definition}

\begin{document}
\pagebreak


\title{\textbf{TRANSFORMATION OF THE DISCRETE STABLE PROCESS VIA BRANCHING REPRODUCTION ENVIRONMENT}\thanks{This study is financed by the European Union's NextGenerationEU, through the National Recovery and Resilience Plan of the Republic of Bulgaria, project No BG-RRP-2.004-0008. }}

\vspace{2cm}

\date{}

\author{M. Slavtchova-Bojkova$^{*,**},$ \and P. Mayster$^{**}$}

\maketitle


\vspace{1.5cm}
\begin{abstract}

The current paper focuses on studying  the impact of immigration with an infinite mean, driven by a discrete-stable compound Poisson process, when it is entering the branching environment with infinite variance of reproduction. Our goal is to determine the explicit form of the probability generating function  and subsequently to analyze the probability of extinction, aiming to understand the long-term behavior of such processes.

\end{abstract}

\vspace{1cm} {\bf Keywords:} branching processes, discrete-stable  distribution, Linnik distribution, homogeneous immigration, asymptotic behaviour
\vspace{1cm}

\section{Introduction}

In this study, we aim to determine the exact distribution of the number of living cells or individuals in a population described by a continuous-time Markov branching process (MBP) with homogeneous in time immigration. Additionally, we assume that immigration follows a discrete stable process, characterized by its probability generating function (p.g.f.), with the number of arriving individuals governed by a given probability measure. Once immigration occurs within the ``branching environment", the immigrants begin to evolve according to the MBP governing laws.

Within this framework, our central question is: How does the ``branching environment" transform the immigration process? Furthermore, what conclusions can be drawn about the immigration process after undergoing this ``branching transformation", particularly when our focus is on deriving the exact p.g.f. of the population size?

Later, we will use these results to analyze the probability of extinction, with the aim of understanding the long-term behavior of these processes.

Continuous-time branching processes with immigration were initially introduced by \textsc{Sevastyanov}  in his influential paper $^{\small{\cite{Se}}}$.   
 In these models, immigration occurs through a homogeneous Poisson process, where new individuals arrive randomly at its jump points. The immigrants then independently evolve according to a continuous-time MBP.

Numerous extensions of branching processes with immigration have since been developed and thoroughly explored. Seminal reviews by \textsc{Sevastyanov} $^{\small{\cite{Sevastyanov_68}}}$, and \textsc{Vatutin} and \textsc{Zubkov} $^{\small{\cite{Vatutin_AZ_87, Vatutin_AZ_93}}},$ have highlighted many key results. More recent advancements have been made by \textsc{Barczy et al.} $^{\small{\cite{Barczy1,Barczy2}}}$, \textsc{González et al.}  $^{\small{\cite{Gonz1}}}$, and \textsc{Li et al.}  $^{\small{\cite{Li}}}$, among others, who have expanded the theory of these processes.

This paper introduces a novel approach for explicitly deriving the p.g.f., a task that is generally challenging in the theory of branching processes, for a specific class of  MBP with homogeneous immigration (Theorem 1). Then, by applying Laplace transforms to suitably normalized processes, Theorem 2 establishes detailed limit results based on the relationship between immigration and branching reproduction rates.

\subsection{Model of the branching reproduction with immigration}

Let $X(t), t\geq 0,$ be a MBP with infinitesimal generating function $ f(s)=K(h(s)-s), |s|\leq 1,$ starting with one particle as initial condition. The branching mechanism (the reproduction of particles) is defined by the p.g.f. $h(s), |s|\leq 1$ and exponential life-time of particles with parameter $K > 0$. The immigration flux of particles is defined by the compound Poisson process $ S(t), t\geq 0$.

Let $\{ I_1, I_2, \dots\}$ denote the sequence of independent identically distributed random variables representing the number of immigrating particles, arriving in the branching system by the exponentially distributed  intervals with parameter $\theta >0$. This way, the flow of immigration is defined by a compound Poisson process with p.g.f.
$$ G(t,s)=\exp\{-\theta t(1-g(s))\}, \quad g(s)=\mathbb{E}[s^I], \quad I = I_n, \quad n = 1, 2, \dots.$$
Then branching process with immigration $Y(t), t\geq 0,$ is defined as follows:
$$Y(t)=\sum ^{N(t)}_{i=1}\sum^{I_i}_{j=1}X_{i,j}(t-\tau _i), \quad \mbox{where} \quad
\mathbb{E}[s^{N(t)}] = \exp\{-\theta t(1-s)\},$$
the sequence $ \{\tau_1,\tau_2, \dots \}$ describes the arrival times of immigrants and inter-arrival times intervals $ (\tau_n-\tau_{n-1}), n = 2, 3, \dots$ are exponentially distributed with parameter $\theta>0$ and $ X_{i,j}(t)$ are independent copies of the MBP without immigration $X(t)$. The parameter $\theta>0$ controls the intensity of immigration.

\subsection{Immigration with infinite mean and reproduction with infinite variance}
We consider the immigration flux of particles with infinite mean defined by the discrete-stable process $S(t), t\geq 0,$ with p.g.f.
\begin{equation}    \label{S}
 G(t, s) = \exp\left\{-\theta t (1-s)^\gamma \right\} ,\quad |s|\leq 1,\quad 0<\gamma<1.
\end{equation}
The infinitesimal generating function for the p.g.f. $G(t,s)$ is denoted by
 $$\varphi(s)=\theta(1-g(s)), \quad {\mbox{so that}} \quad G(t,s)=e^{-t\varphi (s)}.$$
The number of immigrating particles is defined by the L\'{e}vy measure $ \Pi(k)$
$$ \Pi(k)=P(I=k)=(-1)^{k-1}\frac{[\gamma]_{k\downarrow}}{k!},\quad 0 < \gamma < 1,\quad k = 1, 2, \dots$$
 given by Sibuya p.g.f.,  see \textsc{Sibuya} $^{\small{\cite{Sib}}}$,
$$g(s)=1-(1-s)^\gamma, \quad g(s)=1- \sum^\infty_{k=0}\frac{s^k (-1)^k [\gamma]_{k\downarrow}}{k!}= \sum^\infty_{k=1}\frac{s^k (-1)^{k-1} [\gamma]_{k\downarrow}}{k!},$$
where $[\gamma]_{k\downarrow}= (\gamma)(\gamma-1)...(\gamma-k+1).$

We study the transformation of a flux of immigrating particles entering domain of branching reproduction.
Suppose, the branching mechanism with infinite variance, is defined by the p.g.f.
of the newly-born particles given by
\begin{equation}  \label{offspring_pgf}
h(s)=s+\frac{(1-s)^{1+\beta}}{1+\beta}, \quad h(0)=\frac{1}{1+\beta}<1, \quad h(1)=1, \quad 0<\beta<1.
\end{equation}
In this case the infinitesimal generating function for $X(t),$  defined by the branching mechanism $h(s)$ and exponential life-time of particles with parameter $K>0,$ takes the form
$$f(s)=K(h(s)-s)=K\frac{(1-s)^{1+\beta}}{1+\beta}.$$
The backward Kolmogorov equation in this case is   as follows
$$\frac{d}{dt}(F(t,s)))=\frac{K}{1+\beta}(1-F(t,s))^{1+\beta},\quad F(0,s)=s.$$
It has the following explicit solution
\begin{equation} \label{explicit_solution}
 1 - F(t,s)= (1-s)\left\{ 1+\frac{K \beta (1-s)^\beta t}{1+\beta}\right\}^{-1/\beta}.
\end{equation}
 The partial derivatives $F'_{t} $ and $ F'_{s}$ satisfy the forward Kolmogorov equation.
 The second derivative $ F''_{ss}$ in the neighborhood of the point
 $ s=1$ describes the infinite variance of reproduction
$$ F''_{ss}\sim \frac{1}{(1-s)^{1-\beta}},\quad s\rightarrow 1,\quad 0<\beta<1. $$
The infinite mean of immigration is defined by the first derivative of the p.g.f. $G(t,s)$
$$G'_s(t,s)= e^{-\theta t(1-s)^\gamma}\left(\frac{\theta t}{(1-s)^{1-\gamma}}\right)\sim \frac{1}{(1-s)^{1-\gamma}},\quad s\rightarrow 1,\quad 0<\gamma<1. $$
For more  details see the book of  \textsc{Sevastyanov}    $^{\small{\cite{S}}},$ p.44, and
   \textsc{Steutel et al.} $^{\small{\cite{Ste}}},  ^{\small{\cite{SW}}}, ^{\small{\cite{SH}}}$.

 The transformation of the immigration flux of particles $Y(t)$ is studied by the p.g.f.
$\Phi (t,s) = \mathbb{E}\left[s^{Y(t)}\right] $, represented by the integral,
 \begin{equation}\label{integral_repr}
\Phi (t,s)= \exp\{-\theta \int^t_{u=0}(1-F(u,s))^\gamma du\},\quad \Phi (0,s)= 1,
\end{equation}
where $F(t,s)$ is satisfying (\ref{explicit_solution}).

The equality of the parameters $\gamma=\beta $ is resulting in  the explicit solution $\Phi (t,s)$ of equation (\ref{integral_repr})  being the p.g.f. of discrete Linnik distribution for fix time $t>0$.
The asymptotic limit distribution is represented by one-sided positive Linnik distribution, see
\textsc{Christoph  et al.} $^{\small{\cite{Chr}}}, ^{\small{\cite{ChrS}}}$.
 When $\gamma=\beta \rightarrow 1,$ the branching reproduction approaches the linear  birth-death process, and discrete-stable process becomes homogeneous Poisson process. Respectively, discrete Linnik distribution is transformed into the Negative-Binomial and $R_{+}$ -valued Linnik distribution becomes Gamma distribution, in  full compliance with the classical theory.

 The inequality of the parameters $\gamma\neq\beta $ leads to the explicit solution of (\ref{integral_repr})  $\Phi (t,s)$ represented by the Gauss hypergeometric function. Its series expansion shows the first approximation as p.g.f. of discrete-stable distribution with  order $ 0<\gamma<1$. The proper asymptotic limit distribution exists for the normalised process $Z(t)$ only when
$$0<\gamma<\beta\leq 1,\quad Z(t)=\frac{Y(t)}{(t/A)^{1/\gamma}},\quad A=\frac{1+\beta}{K\beta},\quad t\rightarrow \infty. $$
All other parameter configurations and normalization lead to zero or infinity degeneration.

 Linear forms and statistical criteria for the Linnik distribution started in  1953  with \textsc{Linnik} article $ ^{\small{\cite{L}}}$.
Due to \textsc{Devroye} $^{\small{\cite{D}}}$ in 1990 a non-negative integer valued discrete Linnik random variables are defined by their p.g.f.
 The  stable distribution and its applications are developed in the book of \textsc{Uchaikin et al.} $^{\small{\cite{UZ}}}$
starting with the articles of \textsc{Zolotarev} $^{\small{\cite{Z}}}, ^{\small{\cite{Zo}}}$. The analogue of the discrete
 self-decomposability and stability is defined in \textsc{Steutel et al.} $^{\small{\cite{Ste}}}, ^{\small{\cite{SH}}}$.
 The recent article by \textsc{Mitov et al.}   $^{\small{\cite{MY}}}$  develops the case with non-homogeneous immigration.

\section{The p.g.f. of MBP with immigration defined by the discrete-stable process}
\begin{thm}
Let $ X(t)$ be a critical MBP with  branching mechanism p.g.f., given by (\ref{offspring_pgf})
 and the solution $F(t,s)$ of the Kolmogorov equation 
is given by (\ref{explicit_solution}).
If the flow of immigration is defined by the discrete-stable process with p.g.f. (\ref{S})
 then the number $Y(t)$ of particles alive at time $t>0$ has the following p.g.f. 

(i) If  $\beta = \gamma$ then
$$ \Phi(t,s) = \left(\frac{1}{1+\frac{t(1-s)^\beta}{A}}\right)^{\theta A}:=\left(\frac{1}{B(t,s)}\right)^{\theta A},$$
 where
\begin{equation} \label{B_function}
 A=\frac{1+\beta}{K\beta},\quad B(t,s) = 1+\frac{t(1-s)^\beta}{A},\quad B:=B(t,s);
\end{equation}

(ii) If $\beta \neq \gamma$ then
\begin{equation}    \label{P}
\Phi(t,s)= \exp\left\{\frac{-\theta A(1-s)^\gamma}{(1-s)^\beta }\left(\frac{B^{1-\delta}-1 }{1-\delta }\right)\right\},\quad \delta =\frac{\gamma}{\beta}>0,
\end{equation}
or in series expansion

\begin{equation}    \label{se}
 \Phi (t,s)= \exp\left\{-\theta t(1-s)^\gamma\left(1+\sum ^\infty_{j=2}\left(\frac{(1-s)^\beta t} {A}\right)^{j-1} \frac{[-\delta]_{({j-1})\downarrow}}{j!}\right) \right\},
\end{equation}
where  the relation between the increasing and decreasing factorials,
\begin{equation} \label{fac}
 [1-\delta]_{j\downarrow} =(1- \delta)[-\delta]_{(j-1)\downarrow}=-(\delta-1)(-1)^{j-1}[\delta]_{(j-1)\uparrow}.
\end{equation}
For any fix time  $t<\infty, \lim_{s\rightarrow 1}\Phi (t,s)=1.$
\end{thm}

\begin{proof}
 The immigration flux of particles, entering the domain of branching has the p.g.f. (\ref{integral_repr})
 $$\Phi (t,s)= \exp\{-\theta \int^t_{u=0}(1-F(u,s))^\gamma du\},\quad \Phi (0,s)= 1.$$
 Following (\ref{explicit_solution}) and (\ref{B_function}), after the change of variable it is written, in the form
  $$y=1+\frac{u(1-s)^\beta}{A},\quad du=\frac{A dy}{(1-s)^\beta},\quad (1-F(u,s))^\gamma=\frac{(1-s)^\gamma }{y^{\gamma/\beta}}.$$
To calculate the integral, we must distinguish two cases, $\beta=\gamma$ and  $\beta\neq \gamma$.

(i) Let
$$ \beta=\gamma,\quad \Phi (t,s)= \exp\left\{\frac{-\theta A(1-s)^\gamma}{(1-s)^\beta }\int^B_{y=1}\frac{dy}{y}\right\}= \left(\frac{1}{B}\right)^{\theta A}. $$
(ii) Now, we calculate for $\beta\neq \gamma, $ and $ \delta=\gamma/\beta >0$, the following
$$\Phi (t,s)=\exp\left\{\frac{-\theta A(1-s)^\gamma}{(1-s)^\beta }\int^B_{y=1}\frac{dy}{y^\delta}\right\}
= \exp\left\{\frac{-\theta A(1-s)^\gamma}{(1-s)^\beta }\left(\frac{B^{1-\delta}-1 }{1-\delta }\right)\right\} . $$
The series expansion of the ``positive binomial" or negative binomial, converges for $$0<\frac{t(1-s)^\beta}{A} <1,\quad 0<\beta<1.$$
The relation between the increasing and decreasing factorial (\ref{fac}) gives the following,
 $$1-\delta>0,\quad   B^{1-\delta}=\left(1+\frac{t(1-s)^\beta}{A}\right)^{1-\delta}=1+\sum ^\infty_{j=1}\left(\frac{(1-s)^\beta t} {A}\right)^j \frac{[1-\delta]_{j\downarrow}}{j!} $$
 and
  $$\delta-1>0,\quad   B^{1-\delta}=\left( \frac{1}{1-\frac{-t(1-s)^\beta}{A}}\right)^{\delta-1}=1+\sum ^\infty_{j=1}\left(\frac{-(1-s)^\beta t} {A}\right)^j \frac{[\delta-1]_{j\uparrow}}{j!}.  $$
Consequently,
$$B^{1-\delta}-1=\left(\frac{(1-s)^\beta t} {A}\right)(1-\delta)\sum ^\infty_{j=1}\left(\frac{(1-s)^\beta t} {A}\right)^{j-1} \frac{[-\delta]_{(j-1)\downarrow}}{j!}. $$
Then, the time parameter $t>0$, appears as multiple of $(1-s)^\gamma$ in exponent (\ref{P}) after the series expansion and integration,
$$\Phi (t,s)= \exp\left\{-\theta t(1-s)^\gamma\left(1+\sum ^\infty_{j=2}\left(\frac{(1-s)^\beta t} {A}\right)^{j-1} \frac{[-\delta]_{(j-1)\downarrow}}{j!}\right) \right\} . $$
\end{proof}
 \begin{lem}
The p.g.f. $$\Phi (t,s)= \exp\left\{\frac{-\theta A(1-s)^\gamma}{(1-s)^\beta }\left(\frac{B^{1-\delta}-1 }{1-\delta }\right)\right\} $$
and its derivative yield the equation in partial derivatives,
$$ \Phi'_t(t,s)=-\varphi(s) \Phi(t,s)+f(s)\Phi'_s(t,s),\quad \varphi(s)=\theta(1-s)^\gamma,\quad f(s)=\frac{(1-s)^{1+\beta}}{A \beta}. $$

\end{lem}

\section{Asymptotic behavior}
It is clear that the process $Y(t)$ has some infinite mean for $t>0$. To obtain any asymptotic limit for $ t\rightarrow \infty$ we must normalise $ Y(t)$ taking into account the behaviour of the extinction probability. The non-extinction probability for MBP $X(t)$ is comparable to the extinction probability of branching process with immigration $Y(t)$,
$$1-F(t,0)=\left(\frac{1}{1+\frac{t}{A}}\right)^{1/\beta}, $$
$$\beta=\gamma,\quad  \Phi (t,0)=\left(\frac{1}{1+\frac{t}{A}}\right)^{\theta A} ,$$
$$\beta\neq \gamma,\quad   \Phi (t,0)= \exp\left\{-\frac{\theta A}{1-\delta}\left(\left(1+\frac{t}{A}\right)^{1-\delta}-1\right) \right\} . $$
We must
 normalise $Y(t)$ with $$ z(t)=\left(\frac{A }{t}\right)^{1/\beta}, \mbox{or} \quad  z(t)=\left(\frac{A }{t}\right)^{1/\gamma}.$$
 We consider the process
 $$ Z(t)=Y(t)z, \quad z=z(t),\quad \lim_{t\rightarrow \infty}z(t)=0. $$
The Laplace transform of the normalized process $Z(t)$ is 
$$\Psi(t,\lambda)=\mathbb{E}[\exp \{-\lambda Z(t)\}]=\mathbb{E}[\exp \{-\lambda z \}]^{Y(t)}, \Psi(t,\lambda)= \Phi (t,e^{-\lambda z}), \lim_{t \to \infty} \Psi(t,\lambda)=\Psi(\lambda).$$
Let $ s=e^{-\lambda z}$, then for $ 1-s\rightarrow 0, \quad  t\rightarrow \infty$, $ 1-s\sim \lambda z,$ and we have
$$(1-s)^\gamma=(1-e^{-\lambda z})^\gamma \sim \lambda^\gamma z^\gamma,\quad (1-s)^\beta=(1-e^{-\lambda z})^\beta \sim \lambda^\beta z^\beta.$$
\begin{thm}
Let the processes $ X(t)$ and $ Y(t)$ be defined by (\ref{offspring_pgf}), (\ref{explicit_solution}) and (\ref{S}) as in the previous theorem.
Then the asymptotic limit of the normalized process $Z(t)$ exists following the parameter configuration and normalisation as follows,
$$1) \quad 0< \gamma=\beta\leq 1,\quad \frac{t}{A}=z^{-\beta},\quad \Psi(\lambda)=\left(\frac{1}{1+\lambda^\beta}\right)^{\theta A},\quad \lambda>0, $$
$$2) \quad 0< \gamma<\beta\leq 1,\quad \frac{t}{A}=z^{-\gamma},\quad \Psi(\lambda)=\exp\left\{-\theta A \lambda^\gamma \right\},\quad \lambda>0, $$
$$3)  \quad 0< \gamma<\beta\leq 1,\quad \frac{t}{A}=z^{-\beta}, \quad \Psi(\lambda)=0,\quad \lambda>0, $$
$$4) \quad 0< \beta<\gamma\leq 1,\quad \frac{t}{A}=z^{-\beta},\quad \Psi(\lambda)=1,\quad \lambda>0.$$
$$5) \quad 0< \beta<\gamma\leq 1,\quad \frac{t}{A}=z^{-\gamma}, \quad \Psi(\lambda)=1,\quad \lambda>0, $$

\end{thm}

\begin{proof}
$1)$ \textit{The first point }is obvious,
$$\Psi(t,\lambda)= \left( 1+\frac{ (1-e^{-\lambda z})^\beta t} {A}\right)^{-\theta A},\quad \lim_{t\rightarrow \infty }\Psi(t,\lambda)=\left(\frac{1}{1+\lambda^\beta}\right)^{\theta A}. $$
$2)$ \textit{The second point,} $ \delta=\frac{\gamma}{\beta}<1$.
In this normalisation, $ \beta-\gamma>0$, the series expansion (\ref{se}) converges to a bounded sum, so that we derive the following representation
 $$ B^{1-\delta}=\left(1+\frac{t(1-s)^\beta}{A}\right)^{1-\delta}=1+\sum ^\infty_{j=1}\left(\frac{(1-s)^\beta t} {A}\right)^j \frac{[1-\delta]_{j\downarrow}}{j!}, \quad 1-\delta>0 .  $$
Then substracting the first multiple before the sum in the series expansion,  and taking into account $\lambda^\beta z^{\beta-\gamma} (-\delta)\rightarrow 0,$ when $ z\rightarrow 0, $
 $$\Psi(t, \lambda) = \exp\left\{-\theta A \lambda^\gamma\left(1+ \lambda^\beta z^{\beta-\gamma} (-\delta)\left\{1+\sum ^\infty_{j=3}\left(\lambda^\beta z^{\beta-\gamma}\right)^{j-2} \frac{[-\delta-1]_{{j-2}\downarrow}}{j!} \right\}\right) \right\}. $$
 Consequently,
$$\Psi(\lambda)=\lim_{t\rightarrow \infty}\Phi (t,e^{-\lambda z})=\exp\left\{-\theta A \lambda^\gamma\right\}. $$

$3)$ \textit{The third point,} $ \delta=\frac{\gamma}{\beta}<1$ and the normalisation $z^\beta=\frac{A}{t}$
 implies for the function $B(t,s)\sim 1+\lambda^\beta$ .
 The inequality $ \gamma -\beta <0$, leads to the limit
  $$z\rightarrow 0,\quad (1-s)^{\gamma-\beta}\sim z^{\gamma-\beta}\rightarrow \infty. $$
  Obviously,
$$\lim_{t\rightarrow \infty}\Phi (t,s)=\lim_{t\rightarrow \infty} \exp\left\{-\theta A(1-s)^{\gamma-\beta}\left(\frac{B^{1-\delta}-1 }{1-\delta }\right)\right\}=0. $$
$4)$ \textit{The fourth point,}
when $0<\beta<\gamma\leq 1 $, then $ \delta=\gamma/\beta>1$. We have
$$\Phi (t,s)= \exp\left\{\frac{-\theta A(1-s)^{\gamma-\beta}}{\delta-1}\left(1-\frac{1}{B^{\delta-1}}\right)\right\},\quad \delta>1,\quad \gamma-\beta>0 . $$
The normalisation, $t=Az^{-\beta}$, implies
 $$ z^{\gamma-\beta}\rightarrow 0, \quad B^{\delta-1}\sim (1+z^{-\beta}\lambda^\beta z^\beta)^{\delta-1}=(1+ \lambda^\beta  )^{\delta-1}.   $$
 Consequently,
 $$\Psi(\lambda)=\lim_{t\rightarrow \infty}\Psi(t,\lambda)= 1,\quad \lambda>0. $$

$5)$ \textit{The fifth point,}
 Take, $t=Az^{-\gamma}$, $ \beta-\gamma<0,$
then $ z^{\beta-\gamma}\rightarrow \infty $, and
$$t=Az^{-\gamma},\quad B^{\delta-1}\sim(1+z^{-\gamma}\lambda^\beta z^\beta)^{\delta-1}=(1+ \lambda^\beta z^{\beta-\gamma}  )^{\delta-1}\rightarrow \infty .  $$
 Consequently,
  $$(1-s)^{\gamma-\beta} \rightarrow 0,\quad 1 -\frac{1}{B^{\delta-1}}\rightarrow
1 $$
  and
  $$\Phi (t,s)= \exp\left\{\frac{-\theta A(1-s)^{\gamma-\beta}}{\delta-1}\left(1-\frac{1}{B^{\delta-1}}\right)\right\},\quad \delta>1,\quad \gamma-\beta>0 , $$

 $$\lim_{t\rightarrow \infty} \Psi(t,\lambda)= 1,\quad \lambda>0. $$

\end{proof}
\section{Conclusion}

 The study delves into the probability generating functions, a tool often used in analyzing branching processes to understand the distribution of the number of individuals at a given time.
There's also interest in the power of the mean and variance. While the variance is infinite, suggesting heavy-tailed behavior or high variability in outcomes, the process remains discrete-stable in first approximation of the series expansion, implying a form of stability in the process's distribution even with extreme variability.

The explicit solution and   limit  results completely comply with the classical theory of \textsc{Sevastyanov}  $^{\small{\cite{S}}}$
and those for non-homogeneous immigration derived  by   \textsc{Mitov et al.}  $^{\small{\cite{MY}}}$.

\begin{center}
\date{

* Department of Probability, Operations Research and Statistics\\
Faculty of Mathematics and Informatics\\
Sofia University "St. Kl. Ohridski"\\ 5, blvd. J. Bourchier,
1164 Sofia, BULGARIA}\\

** Department of Operations Research, Probability and Statistics\\
Institute of Mathematics and Informatics\\ Bulgarian Academy of
Sciences \\1113 Sofia,
BULGARIA \\
\end{center}
\end{document}